\documentclass[a4paper,reqno,oneside,10pt]{amsart} 

\usepackage[frenchb]{babel}
\usepackage[latin1]{inputenc}
\usepackage{epsfig}
\usepackage{graphicx}
\usepackage{psfrag}

\usepackage{amsmath}
\usepackage{amssymb}
\usepackage{enumerate}
\usepackage{mathrsfs}

\newcommand{\C}{\mathbf{C}}
\newcommand{\R}{\mathbf{R}}

\newcommand{\Z}{\mathbf{Z}}

\def\g{\gamma}

\def\d{\delta}
\def\e{\varepsilon}

\newcommand{\sys}{\mathsf{sys}}

\newcommand{\area}{\mathsf{aire}}
\newcommand{\ri}{\mathsf{R}}

\newcommand{\diff}{\mathrm{d}}

\newcommand{\mrm}{\mathrm}

\theoremstyle{plain}
\newtheorem{theorem}{Th\'eor\`eme}[section] 
\newtheorem*{theoremnonumber}{Th\'eor\`eme}
\newtheorem*{theorem1}{Th\'eor\`eme~1}
\newtheorem*{theorem2}{Th\'eor\`eme~2}

\newtheorem*{theorem4}{Th\'eor\`eme~4}

\newtheorem{proposition}[theorem]{Proposition}

\newtheorem{lemma}[theorem]{Lemme}

\newtheorem*{lemmedeschwarz}{Lemme de Schwarz}

\newtheorem{thm-defi}[theorem]{Théorème - Définition}
\newtheorem{pro-defi}[theorem]{Proposition - Définition}

\theoremstyle{definition}

\newtheorem{remark}{Remarque}[section]
\newtheorem{remarks}{Remarques}[section]
\newtheorem*{remarknonumber}{Remarque}

\title[]{Le lemme de Schwarz et la borne supérieure du rayon d'injectivité des surfaces}
\date{Le \today}
\thanks{This work has been fully supported by FIRB 2010 (RBFR10GHHH$_{003}$).}

\begin{document}

\maketitle

\begin{flushleft} 
            \textbf{Matthieu Gendulphe}\\
  \begin{small}Dipartimento di Matematica Guido Castelnuovo\\
                      Sapienza università di Roma,
                      Piazzale Aldo Moro, 00185 Roma\\
                      Matthieu@Gendulphe.com\end{small}
\end{flushleft}

\renewcommand{\abstractname}{Abstract} 
\begin{abstract}
We study the injectivity radius of complete Riemannian surfaces $(S,g)$ with bounded curvature $|K(g)|\leq 1$. We show that if $S$ is orientable with nonabelian fundamental group, then there is a point $p\in S$ with injectivity radius $\ri_p(g)\geq\mrm{arcsinh}(2/\sqrt{3})$. This lower bound is sharp independently of the topology of $S$. This result was conjectured by Bavard who has already proved the  genus zero cases (\cite{bavard-these}). We establish a similar inequality for surfaces with boundary.\par
 The proofs rely on a version due to Yau (\cite{yau}) of the Schwarz lemma, and on the work of Bavard (\cite{bavard-these}). This article is the sequel of \cite{schwarz} where we studied applications of the Schwarz lemma to hyperbolic surfaces.
\end{abstract}

\section{Introduction}\label{sec:introduction} 

\subsection{\'Enoncé des principaux résultats} Nous savons, depuis les travaux de Poincaré et K\oe be sur l'uniformisation, que toute métrique riemannienne sur une surface est conforme à une métrique à courbure constante. Nous sommes alors tentés d'attribuer à ces dernières des propriétés d'extrémalité relatives à certains invariants métriques.\par
 Dans cette direction, C.~Bavard a montré (\cite{bavard-these,bavard84})~:
 
\begin{theoremnonumber}[Bavard] Si $g$ est une métrique à coubure $|K(g)|\leq 1$ sur la sphère 2-dimensionnelle (resp. sur le plan projectif réel), alors il existe un point $p$ de la sphère (resp. du plan projectif) dont le rayon d'injectivité $\ri_p(g)$ est supérieur ou égal à $\pi$ (resp. à $\pi/2$). 
\end{theoremnonumber} 

\begin{remarknonumber}
Ces bornes sont optimales, et atteintes par les métriques à courbure constante égale à $1$.
\end{remarknonumber}

Pour les surfaces ne supportant pas de métrique à courbure positive ou nulle, Bavard a conjecturé (voir \textsection~3.6 de \cite{bavard-these}) le résultat suivant que nous établissons dans cet article~:

\begin{theorem1} Soit $S$ une surface orientable sans bord dont le groupe fondamental n'est pas abélien. Si $g$ est une métrique complète à courbure $|K(g)|\leq 1$ sur $S$, alors il existe un point $p\in S$ dont le rayon d'injectivité $\ri_p(g)$ est supérieur ou égal à $\mrm{arcsinh}(2/\sqrt{3})\approx 0,97$.\par
 Cette borne est optimale quel que soit le type topologique de $S$. De plus, si la borne supérieure du rayon d'injectivité de $g$ est égale à $\mrm{arcsinh}(2/\sqrt{3})$, alors $S$ est homéomorphe à la sphère privée de trois points.
\end{theorem1} 

 Le cas des surfaces orientables de genre nul est déjà connu (\cite{bavard-these} \textsection~2 et 3). Pour les surfaces orientables de genre positif, Bavard (\cite{bavard-these} \textsection~3.6) a montré l'existence d'un point dont le rayon d'injectivité est supérieur ou égal à  $\frac{\ln 3}{2}\approx 0,55$. Le corollaire~1.11 de l'article \cite{bavard-pansu} de C.~Bavard et P.~Pansu donne une estimée légèrement moins bonne.\par

 La version hyperbolique du théorème~1 est due à A.~Yamada (\cite{yamada,schwarz})~:

\begin{theoremnonumber}[Yamada]
 Soit $(S,g_0)$ une surface hyperbolique orientable sans bord. La borne supérieure du rayon d'injectivité de $g_0$ est supérieure ou égale à $\mrm{arcsinh}(2/\sqrt{3})$, avec égalité si et seulement si $(S,g_0)$ est isométrique au pantalon hyperbolique à trois pointes. 
\end{theoremnonumber}
 
\begin{remarknonumber}
La borne de Yamada est optimale quel que soit le type topologique de la surface, mais elle n'est atteinte que pour la sphère privée de trois points.
\end{remarknonumber}
 
 Nous nous intéressons aussi aux surfaces à bord géodésique pour lesquelles le même problème présente un intérêt. Dans le cas du disque, la question a été résolue indépendamment par Y.~Burago (\cite{burago}) et Bavard (\cite{bavard84})~:

\begin{theoremnonumber}[Burago, Bavard]
Soit $g$ une métrique sur le disque fermé telle que le bord soit géodésique. Si $g$ est à courbure $|K(g)|\leq 1$, alors il existe un point $p$ du disque dont le rayon d'injectivité $\ri_p(g)$ est supérieur ou égal à $\pi/2$. 
\end{theoremnonumber} 

\begin{remarknonumber}
Cette borne est optimale, et atteinte par une hémisphère à courbure constante égale à 1.
\end{remarknonumber}
 
 Pour les surfaces n'admettant pas de métrique à courbure positive ou nulle, nous généralisons un résultat bien connu de la géométrie hyperbolique~:
 
\begin{theorem2}
Soit $S$ une surface dont les composantes de bord éventuelles sont compactes, et dont le groupe fondamental n'est pas virtuellement abélien. Soit $g$ une métrique sur $S$ telle que $\partial S$ est géodésique. Si $g$ est à courbure bornée $|K(g)|\leq 1$, alors la surface $(S,g)$ contient $-\chi(S)$ disques métriques de rayon $\frac{\ln 3}{2}$ disjoints et homéomorphes au disque unité.
\end{theorem2}
  
\subsection{Plan}
La plupart des résultats de cet article reposent sur le lemme de Schwarz. Dans le \textsection~\ref{sec:minoration} nous rappelons la version de S.T.~Yau (\cite{yau}) du lemme de Schwarz, puis nous prouvons le théorème~1. Sa démonstration fait appel à de nombreux résultats provenant de la thèse de Bavard (\cite{bavard-these}). Dans le \textsection~\ref{sec:bord} nous démontrons le théorème~2, ceci nécessite un lemme d'approximation très utile (lemme~\ref{lem:technique}), qui permet de se ramener au cas sans bord.\par
 Les \textsection~\ref{sec:majoration} et \ref{sec:applications} sont de nature un peu différente. Dans le \textsection~\ref{sec:majoration} nous majorons la borne supérieure du rayon d'injectivité des surfaces à courbure $0\geq K\geq -1$. L'idée consiste à adapter des arguments dus à M.~Katz et S.~Sabourau (\cite{sabourau}). Dans le \textsection~\ref{sec:applications} nous minorons la longueur des géodésiques fermées non simples des surfaces à courbure $K\geq -1$, nous répondons ainsi à une question de P.~Buser. 
 
\subsection{Conventions}
Les métriques considérées dans cet article sont riemanniennes. Une métrique \emph{hyperbolique} est une métrique complète à courbure constante $-1$. Une application conforme entre deux surfaces riemanniennes est une application dont la différentielle est en tout point une similitude, nous n'autorisons pas les points singuliers. Toutes les surfaces sont supposées connexes.

\subsection{Remerciements} Je remercie Christophe Bavard pour les différents échanges que nous avons eus. Ce travail a été inspiré par ses résultats non publiés (\cite{bavard-these}). Je remercie Juan Souto pour une discussion qui m'a permis d'aborder le sujet avec moins de naïveté.

\section{Minoration de la borne supérieure du rayon d'injectivité}\label{sec:minoration}

\subsection{Définitions et notations}   Rappelons que si $p$ est un point d'une surface riemannienne complète $(S,g)$, éventuellement à bord géodésique, alors le rayon d'injectivité en $p$ est égal à la plus petite des trois longueurs suivantes (lemme~5.6 de \cite{cheeger})~: la distance de $p$ au bord éventuel, la distance de $p$ à son lieu des points conjugués, la demi-longueur du plus court lacet géodésique d'origine $p$. Nous notons $\ri_p(g)$ le rayon d'injectivité de $g$ au point $p$, et $\ri$ la borne supérieure du rayon d'injectivité sur $S$.\par
  Par la suite, il sera utile de distinguer les lacets géodésiques contractiles de ceux non contractiles. Nous désignons par $\sys_p(g)$ la longueur du plus court lacet géodésique non contractile d'origine $p$. Cette quantité est souvent appelée \emph{systole de $g$ au point $p$}.\par
 
\subsection{En courbure négative ou nulle}\label{sec:schwarz}
Parmi les nombreux avatars du lemme de Schwarz, nous utiliserons l'énoncé suivant dû à S.T.~Yau (théorème~1 de \cite{yau})~:
 
\begin{lemmedeschwarz}
Soit $(S,g_0)$ une surface riemannienne sans bord à courbure négative pincée $0>\sup_S K(g_0)\geq \inf_S K(g_0)\geq -1$. Soit $g$ une métrique complète sur $S$ dans la même classe conforme que $g_0$. Si $\inf_S K(g)\geq \sup_S K(g_0)$ alors $g\geq g_0$. 
\end{lemmedeschwarz}

\begin{remarks}
\begin{enumerate}
\item Nous ne supposons pas la surface d'aire finie.
\item M.~Troyanov a obtenu un énoncé plus général (voir \cite{troyanov}).
\item Ce \guillemotleft~lemme~\guillemotright\  produit une inégalité sur les distances qui va dans le même sens que celle sur les courbures, à l'inverse du théorème de comparaison de Rauch.
\item Par le même théorème de Yau (théorème~1 de \cite{yau}), la courbure de $g$ prend nécessairement une valeur négative, soit $\inf_S K(g)<0$.
\item Lorsque l'inégalité est stricte $\inf_S K(g)> \sup_S K(g_0)$, le théorème donne
$$g\geq \frac{\sup_S K(g_0)}{\inf_S K(g)} g_0.$$
\end{enumerate}
 \end{remarks}

 Sous les hypothèses du lemme, nous avons $I(g)\geq I(g_0)$ pour tout invariant métrique \emph{croissant} $I$. Ainsi, toute métrique hyperbolique réalise le minimum de $I$ parmi les métriques complètes à courbure $K\geq -1$ dans sa classe conforme. Comme invariants métriques monotones, nous pouvons citer l'aire, l'entropie, la systole, le diamètre, le rayon de recouvrement, le spectre marqué des longueurs.\par
  Le rayon d'injectivité n'est pas un invariant métrique monotone, puisqu'il tient compte des points conjugués, et des lacets géodésiques contractiles. Cependant, si la courbure est négative ou nulle, nous avons $\ri_p=\sys_p/2$ en tout point $p$ de $S$. Comme $\sys_p$ est un invariant métrique croissant, le théorème de Yamada (\textsection~\ref{sec:introduction}) et le lemme de Schwarz donnent immédiatement~:
 
\begin{proposition}
Soit $S$ une surface orientable sans bord non homéomorphe à la sphère, au disque, au cylindre ou au tore. Soit $g$ est une métrique complète sur $S$.
\begin{enumerate}[i)]
\item Si $0\geq K(g)\geq -1$, alors $\ri(g)\geq \mathrm{arcsinh}(2/\sqrt{3})$. 
\item Si $K(g)\geq -1$, alors il existe un point $p\in S$ tel que $\sys_p(g)\geq 2  \mathrm{arcsinh}(2/\sqrt{3})$.
\end{enumerate}
Ces bornes sont optimales quelle que soit la surface $S$ satisfaisant les hypothèses. De plus, le cas d'égalité n'est réalisé que par des métriques sur la sphère privées de trois points. 
\end{proposition} 

\begin{remark} On peut étendre ce théorème au disque ouvert et au cylindre infini en ajoutant l'hypothèse \emph{$g$ est conforme à une métrique hyperbolique}. Le cylindre a une classe conforme euclidienne (celle de $\C^\ast$), toutes ses autres classes conformes (celles des couronnes et du disque épointé) sont hyperboliques. Le disque a une classe conforme euclidienne (celle du plan euclidien), et une classe conforme hyperbolique (celle du disque de Poincaré).\end{remark}

 L'idée de contrôler les invariants métriques monotones par le lemme de Schwarz n'est pas nouvelle. Mentionnons que P.~Su\'arez-Serrato et S.~Tapie (\cite{tapie}) s'en servent pour établir la rigidité entropique des métriques de Yamabe parmi les métriques conformes à courbure sectionnelle négative.\par 

\subsection{En courbure bornée}\label{sec:courbure-bornee}
Nous étendons l'item i) de la proposition ci-dessus aux surfaces riemanniennes à courbure bornée $|K|\leq 1$.

\begin{theorem1}
Soit $S$ une surface orientable sans bord non homéomorphe à la sphère, au disque, au cylindre ou au tore. Si $g$ est une métrique complète sur $S$ à courbure bornée $|K(g)|\leq 1$, alors $\ri(g)\geq\mrm{arcsinh}(2/\sqrt{3})$.\par
 Cette borne est optimale quelle que soit la surface $S$ satisfaisant les hypothèses. De plus, le cas d'égalité n'est réalisé que par des métriques sur la sphère privées de trois points.
\end{theorem1}

\begin{remark}
De manière équivalente, on pourrait supposer \emph{$S$ orientable sans bord avec un groupe fondamental non abélien}.
\end{remark}

 Le théorème~1 se déduit directement du théorème de Bavard énoncé ci-dessous. Selon ce théorème, lorsque la borne supérieure $\ri(g)$ du rayon d'injectivité est petite, le rayon d'injectivité $\ri_p(g)$ est égal à la demi-systole $\sys_p(g)/2$ en tout point $p$ de $S$. Nous sommes ainsi ramenés à un invariant métrique croissant, et nous concluons facilement grâce au lemme de Schwarz et au théorème de Yamada.

\begin{theoremnonumber}[Bavard]
Soit $(S,g)$ une surface riemannienne complète, à courbure majorée $K(g)\leq 1$, et éventuellement à bord géodésique. Si $\ri(g)< \pi/2$, alors il n'existe pas de lacet géodésique simple contractile de longueur inférieure à $2\pi$, et en tout point $p\in S$ le rayon d'injectivité est donné par 
$$\ri_p(g)=\min\left(\frac{\sys_p(g)}{2},d_g(p,\partial S)\right).$$
\end{theoremnonumber}

\begin{remarks}
\begin{enumerate}
\item Le résultat original (proposition~3.2 de \cite{bavard-these}) est plus général. 
\item Ce théorème interviendra de manière fondamentale dans les preuves des théorèmes~2 et 3.
\end{enumerate}
\end{remarks}

Nous indiquons ci-dessous une preuve utilisant les résultats et les idées des articles \cite{bavard84,bavard-pansu} de Bavard et Bavard-Pansu.

\begin{proof} Supposons qu'il existe un lacet géodésique simple $\g$, contractile et de longueur inférieure à $2\pi$. Nous lui appliquons le procédé de raccourcissement de Birkhoff (voir \textsection~1 de\cite{bavard-pansu}), trois situations sont possibles~:
\begin{itemize}
\item la suite de courbes converge vers un point,
\item la suite de courbes tend vers l'infini,
\item la suite de courbes converge vers une géodésique fermée simple.
\end{itemize}
La première situation contredit le lemme d'homotopie de Klingenberg (voir le lemme~4.3 de \cite{abresch}). Dans la deuxième situation, le lacet $\g$ borde d'un côté un disque, et de l'autre un anneau infini. Ainsi $S$ est homéomorphe au plan et, par le théorème principal de \cite{bavard84}, nous avons $\ri(g)\geq \pi$. Dans la troisième situation, nous obtenons une géodésique fermée simple bordant un disque. Selon le corollaire~3 de \cite{bavard84}, il existe un point de ce disque en lequel le rayon d'injectivité est supérieur ou égal à $\pi/2$. Les deux dernières situations contredisent l'hypothèse $\ri(g)<\pi/2$. Ceci prouve la première assertion.\par
 Puisque $K(g)\leq 1$, la distance entre deux points conjugués le long d'une géodésique de $(S,g)$ vaut au moins $\pi$ (théorème de comparaison de Rauch), d'où la deuxième assertion. Le théorème de comparaison de Rauch est en général énoncé pour les surfaces sans bord, il s'étend aux surfaces à bord géodésique grâce au lemme~\ref{lem:technique}.
\end{proof}

\section{Minoration pour les surfaces à bord géodésique}\label{sec:bord}
La présence d'un bord modifie le comportement du rayon d'injectivité, puisque celui-ci tient compte de la distance au bord. Néanmoins nous obtenons un résultat similaire au théorème~1 dans le cas des surfaces à bord géodésique~:

\begin{theorem2}
Soit $S$ une surface dont les composantes de bord éventuelles sont compactes, et dont le groupe fondamental n'est pas virtuellement abélien. Soit $g$ une métrique complète sur $S$ telle que $\partial S$ est géodésique. Si $g$ est à courbure bornée $|K(g)|\leq 1$, alors $S$ contient $-\chi(S)$ points $p_1,\ldots,p_{-\chi(S)}$ tels que
\begin{enumerate}[i)]
\item en chaque $p_i$ le rayon d'injectivité est minoré $\ri_{p_i}(g)\geq\frac{\ln 3}{2}$,
\item les disques de rayon $\frac{\ln 3}{2}$ centrés aux $p_i$ sont disjoints.
\end{enumerate}
\end{theorem2}

\begin{remarks} Nous ne supposons pas la surface orientable, ni la métrique d'aire finie.\end{remarks}

\begin{proof}
Nous décomposons la surface $(S,g)$ en pantalons à bord géodésique (proposition~\ref{pro:pantalons}). Selon le lemme~\ref{lem:pantalons}, chacun de ces pantalons possède un point dont le rayon d'injectivité est supérieur ou égal à  $\frac{\ln 3}{2}$.
\end{proof}

\subsection{Le rayon d'injectivité des pantalons}
Nous appelons \emph{pantalon} la somme connexe de trois disques ouverts ou fermés. Il y a quatre types topologiques de pantalons, classés suivant leur nombre de composantes de bord.\par
 Commençons par rappeler un résultat bien connu~:

\begin{lemma}\label{lem:pantalonhyp}
Si $X$ est une surface hyperbolique, éventuellement à bord géodésique, alors $X$ contient  $-2\chi(X)$ disques de rayon $\frac{\ln 3}{2}$ disjoints et homéomorphes au disque unité.  
\end{lemma}

\begin{proof}
Il suffit de traiter le cas des pantalons hyperboliques, puisque toute surface hyperbolique de caractéristique d'Euler-Poincaré non nulle admet une décomposition en pantalons. Or, tout pantalon hyperbolique peut s'obtenir en collant deux triangles idéaux le long de leurs côtés, et un triangle idéal contient justement un disque de rayon $\frac{\ln 3}{2}$.
\end{proof}
 
 Nous allons étendre ce résultat aux pantalons riemanniens à bord géodésique et à courbure $|K|\leq 1$. La difficulté vient de ce que le lemme de Schwarz ne s'applique pas aux surfaces à bord.\par 
  Si $S$ est une surface à bord non vide, nous notons $\bar S$ son double lisse sans bord. Une métrique à bord géodésique sur $S$ ne se relève pas nécessairement en une métrique \emph{lisse} sur le double $\bar S$. Ce problème de régularité sera contourné grâce à un lemme d'approximation (lemme~\ref{lem:technique}).
 
\begin{lemma}\label{lem:pantalons}
Si $(P,g)$ est un pantalon riemannien complet, à bord géodésique et à courbure $|K(g)|\leq 1$, alors il existe un point $p\in P$ tel que $\ri_p(g)\geq \frac{\ln 3}{2}$.
\end{lemma}

\begin{proof}
Nous notons $\bar P$ le double sans bord de $P$, et $\iota:\bar P\rightarrow \bar P$ l'involution associée. Pour le moment, nous supposons que la métrique $g$ se relève en une métrique lisse $\bar g$ sur $\bar P$.\par
 Nous notons $\bar g_0$ l'unique métrique hyperbolique conforme à $\bar g$. L'involution $\iota$ est une isométrie de $\bar g$, mais aussi de $\bar g_0$ (par unicité de la métrique hyperbolique dans une classe conforme). Ainsi $\bar g_0$ descend en une métrique  hyperbolique à bord géodésique $g_0$ sur $P$. Par le lemme de Schwarz nous avons $\bar g\geq \bar g_0$, donc $g\geq g_0$.\par
 Nous supposons $\ri(g)<\pi/2$, sinon le lemme est trivialement vérifié. L'inégalité $g\geq g_0$ combinée au théorème de Bavard (\textsection~\ref{sec:courbure-bornee}) donne~:
$$\ri_p(g)=\min\left(\frac{\sys_p(g)}{2},d_g(p,\partial P) \right)\geq\min\left(\frac{\sys_p(g_0)}{2},d_{g_0}(p,\partial P)\right)=\ri_p(g_0)$$
en tout point $p\in P$. Nous concluons grâce au lemme précédent.\par
 Si la métrique $g$ ne se relève pas en une métrique lisse sur $\bar S$, alors nous effectuons le raisonnement ci-dessus avec la métrique $(1+\e)g^\e$, où $g^\e$ ($\e>0$) satisfaisant $|K(g^\e)|\leq 1+\e$ est donnée par le lemme~\ref{lem:technique}. Nous trouvons ainsi un point $p^\e\in P$ tel que $\ri_{p^\e}(g^\e)\geq \frac{\ln 3}{2\sqrt{1+\e}}$. Nous concluons en considérant un point limite d'une suite extraite de la famille $\{p^\e\}_{\e>0}$. Notez que $g^\e=g$ en dehors du $\e$-voisinage tubulaire du bord, il existe donc bien des suites extraites convergentes. 
\end{proof}

\subsection{Décomposition en pantalons géodésique}
Soit $S$ une surface dont les composantes de bord sont compactes. Nous appelons \emph{décomposition en pantalons} une famille maximale de courbes fermées simples de $S$ satisfaisant~:
\begin{enumerate}[i)]
\item aucune courbe ne borde un disque, ni un ruban de M\oe bius, ni un anneau~;
\item deux courbes quelconques sont disjointes et non isotopes. 
\end{enumerate}
Une décomposition en pantalons est \emph{géodésique} pour une métrique $g$ sur $S$ si toutes les courbes sont des géodésiques de $g$. 

\begin{proposition}\label{pro:pantalons}
Soit $S$ une surface dont les composantes de bord éventuelles sont compactes, et dont le groupe fondamental n'est pas virtuellement abélien. Si $g$ est une  métrique complète à bord géodésique sur $S$ et à courbure $K(g)\geq -1$, alors $(S,g)$ admet une décomposition en pantalons géodésique.
\end{proposition}

\begin{remarks}\begin{enumerate}
\item Nous ne supposons pas l'aire finie (comparer avec le (4) de \cite{buser}).
\item  Notre démonstration fait appel au lemme de Zorn.
\end{enumerate}\end{remarks}

\begin{proof}
Soit $(\g_i)_{i\in I}$ une famille maximale de géodésiques fermées simples satisfaisant les propriétés i) et ii). Il s'agit de montrer que les composantes de $S\setminus \cup_I \g_i$ sont des pantalons. Par l'absurde, nous supposons qu'une composante $S'$ n'est pas un pantalon. Nous appelons $g'$ la métrique induite par $g$ sur $S'$.\par
 Il existe une courbe fermée simple $c$ de $S'$ satisfaisant la propriété i). Vue comme courbe de $S$, $c$ est disjointe et non isotope aux $\g_i$. Soit $\g$ une géodésique isotope à $c$, donnée par le lemme ci-dessous. Cette géodésique contredit la maximalité de la famille $(\g_i)_I$. \end{proof}

Sous les hypothèses de la proposition, nous avons~:

\begin{lemma}
Soit $c$ une classe d'isotopie de courbe fermée simple ne bordant pas un disque, un anneau, ou un ruban de M\oe bius. La classe $c$ admet un représentant géodésique simple. 
\end{lemma}

\begin{proof}
Nous supposons, pour le moment, que la métrique $g$ se relève en une métrique lisse sur le double sans bord $\bar S$. Sous cette hypothèse, il existe une métrique hyperbolique à bord géodésique $g_0$ sur $S$ vérifiant $g \geq g_0$ (voir la preuve du lemme~\ref{lem:pantalons}).\par
 Soit $\g$ la $g_0$-géodésique dans la classe d'isotopie $c$. Par des formules classiques de trigonométrie hyperbolique, il existe un voisinage compact $K$ de $\g$ tel que tout lacet de $c$ non contenu dans $K$ soit de $g_0$-longueur supérieure à $\ell_g(\g)$. Comme $g\geq g_0$, la $g$-longueur des lacets de $c$ non contenus dans $K$ est elle aussi supérieure à $\ell_g(\g)$. Nous concluons à l'existence d'un représentant de $c$ de $g$-longueur minimale en appliquant le théorème d'Ascoli. Par minimalité, ce représentant est géodésique et simple.\par
 Si $g$ ne se relève pas en une métrique lisse sur le double sans bord, alors nous approchons $g$ par une métrique $g^\e$ satisfaisant $K(g^\e)\geq -1-\e$ (lemme~\ref{lem:technique}). Le raisonnement ci-dessus appliqué à la métrique $(1+\e)g^\e$ produit un compact $K^\e$ et un représentant $\g^\e$ de $c$. Comme $|g-g^\e|\leq \e$ sur le fibré unitaire de $(S,g)$, nous trouvons facilement un compact $K$ tel que tout élément de $c$ non contenu dans $K$ soit de $g$-longueur supérieure à $\ell_g(\g)$. Nous concluons en appliquant le théorème d'Ascoli.
\end{proof}

\subsection{Un lemme d'approximation}
Ce qui suit est très largement inspiré du lemme~3.5 de \cite{bavard-these} et de sa démonstration.

\begin{lemma}\label{lem:technique}
Soit $(S,g)$ une surface riemannienne à bord géodésique non vide et à courbure $K(g)\leq -1$ (resp. $K(g)\geq -1$, resp. $K(g)\leq 1$). Pour tout $\e>0$ il existe une métrique $g^\e$ sur $S$ telle que
\begin{enumerate}[i)]
\item $g^\e$ se relève en une métrique lisse sur $\bar S$,
\item $K(g^\e)\leq -1+\e$ sur $S$ (resp. $K(g^\e)\geq -1-\e$, resp. $K(g)\leq 1+\e$),
\item $g^\e=g$ en dehors de $N_\e(\partial S)$,
\item $|g-g^\e| \leq\e$ sur le fibré unitaire de $N_\e(\partial S)$,
\end{enumerate}
où $N_\e(\partial S)$ désigne le $\e$-voisinage tubulaire de $\partial S$ pour la métrique $g$.   
\end{lemma}

\begin{remarks} 
\begin{enumerate}
\item Si $g$ est complète, alors $g^\e$ est elle aussi complète.
\item Si $|K(g)|\leq 1$ alors $|K(g^\e)|\leq 1+\e$, car la métrique $g^\e$ est construite de la même manière indépendamment du cas considéré.
\end{enumerate} \end{remarks}

\begin{proof}
Soit $\g$ une composante de bord compacte, que nous regardons comme une géodésique d'abscisse curviligne. Soit $\mathsf n$ un champ de vecteur unitaire normal à $\g$. Nous définissons une carte sur un voisinage annulaire $U$ de $\g$ par~:
$$\begin{array}{ccc}
[0,t_0[\times \R/\ell \Z & \longrightarrow & U \\
(t,\theta) & \longmapsto & \exp_{\g(\theta)} (t\mathsf n(\theta))
\end{array}.$$
Nous supposons $t_0>0$ suffisamment petit, et désignons par $\ell$ la longueur de $\g$. Dans cette carte, la métrique $g$ prend la forme
$$g_{(t,\theta)}=\diff t^2+f(t,\theta)^2 \diff \theta^2, $$
où $f$ est une fonction lisse positive sur $[0,t_0[\times \R/\ell\Z$ satisfaisant 
$$\left\{\begin{array}{rcl}
 f(0,\theta) & = & 1,\\
f_t(0,\theta) & = & 0\quad (\partial S \textnormal{ géodésique}),\\
  f_{tt}(t,\theta) & = & -K(t,\theta)f(t,\theta) \textnormal{ pour tout }  (t,\theta).
\end{array}\right.$$
Soit $p$ le polynôme de Taylor de degré deux de $f$ en la variable $t$~:
$$p(t,\theta)=f(0,\theta)+\frac{t^2}{2} f_{tt}(0,\theta)\textnormal{ pour tout }  (t,\theta).$$
Par l'inégalité de Taylor-Lagrange, il existe $M>0$ tel que pour tout $\delta\in]0,t_0[$
$$\left\{\begin{array}{rcl}
\|f-p\|_\infty &\leq & M\delta^3\\
\|f_t-p_t\|_\infty &\leq & M\delta^2\\
\|f_{tt}-p_{tt}\|_\infty &\leq & M\delta\\
\end{array}\right. \textnormal{ pour la norme sup sur } [0,\delta]\times \R/\ell\Z.$$
Fixons une fonction plateau lisse
$$\psi:[0,t_0[\rightarrow [0,1]\textnormal{ telle que } \left\{ \begin{array}{ll}
\psi(t)=0 & \textnormal{si } t\leq \frac{t_0}{4}\\
\psi(t)=1 & \textnormal{si } t\geq \frac{t_0}{2}\\
\end{array}\right. .$$
Pour tout $\delta\in]0,t_0[$ nous posons $\psi_\delta(t)=\psi(t/\delta)$, de sorte que
$$\|\psi'_\delta\|_\infty\leq \frac{\|\psi'\|_\infty}{\delta}\ \textnormal{et}\ \|\psi''_\delta\|_\infty\leq \frac{\|\psi''\|_\infty}{\delta^2}  \textnormal{ pour la norme sup sur }[0,t_0[ .$$
Considérons alors la fonction $F=p +\psi_\delta (f-p)$~; elle vérifie
\begin{eqnarray*}
  F-f              &= & (\psi_\delta-1) (f-p),\\
\|F-f \|_\infty  & \leq & \|f-p \|_\infty   ,\\
\|F-f \|_\infty & \leq & M\delta^3.
\end{eqnarray*}
Et en dérivant par rapport à $t$, il vient~:
\begin{eqnarray*}
 F_{tt} - f_{tt} & = & \psi_\delta'' (f-p)+2 \psi_\d' (f-p)_t+(1-\psi_\d) (f-p)_{tt}, \\
\|F_{tt}-f_{tt}\|_\infty & \leq & 4M\delta.
\end{eqnarray*}
L'hypothèse $K(g)\leq -1$ (resp. $K(g)\geq -1$, resp. $K(g)\leq 1$) équivaut à $f_{tt}\geq f$ (resp. $f_{tt}\leq f$, resp. $f_{tt}\geq -f$), ainsi nous trouvons (en supposant $\delta<1$)
\begin{eqnarray*}
 F_{tt} & \geq &  f_{tt}-4M\delta, \\
 & \geq & f-4M\delta,  \\
 &\geq & F- 5 M\delta\quad (\textnormal{resp. } F_{tt}\leq F+5M\delta, \textnormal{ resp. } F_{tt}\geq -F-5M\delta ).
 \end{eqnarray*}
Comme $f(0,\theta)=1$ pour tout $\theta$, il existe une constante $N>0$ telle que 
$$F\geq 1-N\delta^2 \textnormal{ sur } [0,\delta]\times \R/\ell\Z.$$
\'Etant donné $\e$, nous choisissons $\delta<\e$ tel que
$$\left\{\begin{array}{rcl}
1-N\delta^2 & \geq  & 1/2 \\
\e & \geq & 10M\delta
\end{array}\right. .$$
De cette façon, nous avons $\e F\geq 5M\delta$ et
$$F_{tt}\geq (1-\e) F\quad (\textnormal{resp. } F_{tt}\leq (1+\e) F, \textnormal{ resp. } F_{tt}\geq -(1+\e) F) .$$
Finalement, la métrique $g'=\diff t^2+F(t,\theta)^2\diff \theta^2$ satisfait 
$$\left\{\begin{array}{rcll}
K(g') & \leq & -1+\e & \textnormal{sur}\ [0,\delta t_0[\times \R/\ell\Z,\\
(\textnormal{resp. } K(g') & \geq & -1-\e) & \\
(\textnormal{resp. } K(g') & \leq & 1+\e) & \\
 |g'-g| & \leq & M\delta^3 & \textnormal{sur le fibré unitaire de}\ [0,\delta t_0[\times \R/\ell\Z,
\end{array}\right. $$
et $g'$ se prolonge en une métrique lisse sur le double $]-\delta t_0,\delta t_0[\times \R/\ell\Z$ du cylindre semi-ouvert $[0,\delta t_0[\times \R/\ell\Z$, car $p$ est une fonction paire en la variable $t$. Cette construction locale (au voisinage d'une composante de bord) suffit pour conclure puisque nous avons $g'=g$ sur $]\delta t_0/2,\delta t_0[ \times \R/\ell\Z$.\par
 Si $\g$ est une composante de bord non compacte, alors nous effectuons le même raisonnement en prenant $M$, $N$ et $\delta$ comme fonctions de $\theta$.
 \end{proof}

\section{Majoration de la borne supérieure du rayon d'injectivité}\label{sec:majoration}

 Nous avons minoré la borne supérieure du rayon d'injectivité $\ri$ en supposant la courbure bornée $|K|\leq 1$. Cette hypothèse n'est évidemment pas suffisante pour majorer $\ri$, puisqu'elle ne s'oppose pas à la multiplication par un scalaire plus grand que $1$. Il convient donc d'introduire une condition supplémentaire, la plus naturelle étant la normalisation par l'aire. De manière équivalente, nous considèrerons la quantité homogène $\ri^2/\area$.\par
 Soit $S$ une surface fermée. Sans hypothèse sur la courbure,  $\ri^2/\area$ n'admet pas de borne supérieure finie sur l'espace des métriques riemanniennes sur $S$. En effet, on peut prendre un disque topologique dans $S$, et lui donner la forme d'un cylindre euclidien avec au bout une hémisphère de courbure constante. Ce disque a une aire arbitrairement petite, et le point au centre de l'hémisphère a un rayon d'injectivité arbitrairement grand, pourvu que le cylindre soit suffisamment long et étroit. La métrique construite est de classe $C^1$, mais nous pouvons la lisser tout en préservant la symétrie de révolution, ce qui assure que le rayon d'injectivité ne sera pas beaucoup modifié.\par
 Ainsi, l'existence d'une borne supérieure finie sur $\ri$ requiert des contraintes assez fortes sur la géométrie. Si nous travaillons parmi les métriques à courbure $|K|\leq 1$, alors l'aire d'un disque de rayon $r\leq\pi$ est minorée par $2\pi(1-\cos r)$, d'où
 $$\ri\leq \pi- \arccos\left(\frac{\area}{2\pi}-1 \right)\leq \pi$$
dès que l'aire est inférieure ou égale à $4\pi$. Cette borne est en un sens optimale, puisqu'on peut toujours identifier le bord d'un disque de manière à obtenir la surface topologique souhaitée. Comme le volume minimal de $\R^2$ vaut $2\pi(1+\sqrt{2})$ (théorème de Bavard et Pansu \cite{min-vol}), nous pouvons construire des métriques singulières avec $\ri$ arbitrairement grand dès que l'aire est supérieure ou égale à cette valeur (voir la métrique réalisant le volume minimal de $\R^2$), et cette borne est en un sens optimale.\par
 Puisque courbure et aire bornée ne suffisent pas à contrôler $\ri$, regardons ce qui se passe en courbure négative ou nulle. Si la courbure est constante égale à $-1$, alors il existe une borne optimale sur $\ri$ dépendant de la caractéristique d'Euler-Poincaré (Bavard \cite{bavard}), précisément
$$\cosh \ri\leq \frac{1}{2\sin\g_\chi}\quad\textnormal{avec}\quad \g_\chi=\frac{\pi}{6-6\chi}.$$
Ce résultat s'étend aux surfaces à courbure négative pincée grâce au lemme de Schwarz.
La preuve de Bavard repose sur un argument d'empilement, déjà utilisé dans le calcul de la borne supérieure de la systole sur l'ensemble des surfaces hyperboliques fermées orientables de genre deux (voir Bavard \cite{bavard92} et F.~Jenni \cite{jenni}).
Cet argument a été étendu à la courbure négative ou nulle par M.~Katz et S.~Sabourau (voir le théorème~1.3 de \cite{sabourau}). En effectuant de légers changements dans la preuve de leur théorème, nous obtenons~:

\begin{theorem}
Soit $S$ une surface fermée de caractéristique d'Euler-Poincaré négative. Toute métrique $g$ sur $S$ à courbure négative ou nulle satisfait
$$\frac{\ri(g)^2}{\area(g)} \leq \left[ 6(1-\chi) \tan\left(\frac{\pi}{6(1-\chi)}\right) \right]^{-1}.$$
Il existe des métriques plates à singularités coniques d'angles supérieurs à $2\pi$ (donc à courbure négative ou nulle au sens des espaces CAT(0)) réalisant l'égalité.
\end{theorem}

\begin{remark}
Asymptotiquement cela donne naturellement $\pi \ri^2 \lesssim \area$, ce que nous pouvions obtenir par une application directe du théorème de comparaison de Rauch. Nous observons que le comportement de $\ri$ est nettement différent de celui de la systole (voir \cite{gromov} \textsection~2.C)
\end{remark}

 Décrivons une famille de métriques plates singulières réalisant le cas d'égalité. On part d'un polygone euclidien régulier à $n=6(1-\chi)$ côtés, et l'on identifie les côtés par paires de manière à obtenir la surface $S$ de caractéristique $\chi$ souhaitée. Ces identifications rangent les sommets en cycles de longueur $3$, d'où les angles aux points coniques valent $3\pi (n-2)/n$. Elles ont été largement étudiées, on trouve déjà dans le livre de Fricke et Klein (\cite{klein} p.~267) toutes les identifications donnant la surface orientable de genre deux. Par construction, le rayon du cercle inscrit et l'aire du polygone réalisent l'égalité dans l'inégalité du théorème. Pour certaines identifications, on peut montrer par des calculs  \emph{ad hoc} que le rayon du cercle inscrit dans le polygone est égal au maximum du rayon d'injectivité de la surface plate à singularités coniques.

\begin{proof}
Soit $g$ une métrique riemannienne sur $S$ à courbure négative ou nulle, et soit $p$ un point de $S$ réalisant le maximum du rayon d'injectivité de $(S,g)$. Sur l'espace tangent $T_pS$, nous allons comparer la norme euclidienne $g_p$ avec la métrique $\tilde g$ image réciproque de $g$ par l'exponentielle $\exp_p$. Rappelons que, par le théorème de Cartan-Hadamard, l'exponentielle $\exp_{p}:T_pS\rightarrow S$ définit un revêtement isomorphe au revêtement universel (d'où la notation $\tilde g$).\par
    
   La cellule de Dirichlet-Vorono\"i de l'origine pour la métrique $\tilde g$ 
$$D_0(\tilde g)=\{x\in T_p S~;~  d_{\tilde g}(x,0)\leq d_{\tilde g}(x,\tilde p)\ \textnormal{pour tout relevé}\ \tilde p\in \exp_p^{-1}(p)\} $$
est un disque topologique bordé par une courbe lisse par morceaux composée d'un nombre fini d'arcs. Chaque arc est supporté par une courbe
$$L_{\tilde p}=\{x\in T_p S~;~ d_{\tilde g}(x,0)= d_{\tilde g}(x,\tilde p)\}$$ 
formée des points équidistants de l'origine et d'un relevé $\tilde p\in \exp_p^{-1}(p)\setminus\{0\}$. Nous notons $P$ l'ensemble des relevés $\tilde p$ tels que $L_{\tilde p}$ supporte un côté de $\partial D_0(\tilde g)$. Nous allons comparer la cellule $D_0(\tilde g)$ avec le polygone convexe suivant~:
$$D_0(g_p)=\{x\in T_p S~;~d_{g_p}(x,0)\leq d_{g_p}(x,\tilde p)\ \textnormal{pour tout relevé}\ \tilde p\in P\}.$$
Il s'agit de la cellule de Dirichlet-Vorono\"i de l'origine relativement à l'ensemble $P\cup\{0\}$ pour la norme euclidienne $g_p$. Le choix de l'ensemble $P\cup\{0\}$ assure que le nombre de côtés de $D_0(g_p)$ est inférieur ou égal à celui de $D_0(\tilde g)$.\par

 Selon le théorème de comparaison de Rauch, nous avons 
 $d_{g_p}(\cdot,\cdot) \leq d_{\tilde g}(\cdot,\cdot)$.
 Nous en déduisons $D_0(g_p)\subset D_0(\tilde g)$ et  $\area_{g_p}(D_0(g_p))\leq \area_{\tilde g}(D_0(\tilde g))$. Ainsi
 $$\frac{\ri(g)^2}{\area(g)} = \frac{\ri(g)^2}{\area_{\tilde g}(D_0(\tilde g))} \leq \frac{\ri(g)^2}{\area_{g_p}(D_0(g_p))} .$$
Il reste donc à minorer l'aire du polygone convexe $D_0(g_p)$.\par
 
 Soit $k$ le nombre de côtés du polygone $D_0(g_p)$. Les segments reliant l'origine aux sommets divisent le polygone en $k$ triangles isocèles. Pour chaque triangle, la hauteur issue de l'origine mesure au moins $\ri(g)$, de sorte que l'aire du triangle vaut au moins $\ri(g)^2\tan(\theta/2)$, où $\theta$ désigne l'angle en l'origine. En sommant les aires des triangles, et par convexité de la fonction tangente sur $\R_+$, nous trouvons~:
\begin{eqnarray*}
\area(D_0(g_p)) & \geq & \sum_{i=1,\ldots,k} \ri(g)^2\tan(\theta_i/2), \\
 & \geq & \ri(g)^2 k\tan(\pi/k). 
\end{eqnarray*} 
 Comme la dérivée de la fonction $x\mapsto x\tan(\pi/x)$ sur $[2,+\infty[$ est donnée par
 $$\tan(\pi/x)-\frac{\pi/x}{\cos^2(\pi/x)}= \frac{\sin(2\pi/x)-2\pi/x}{2\cos^2(\pi/x)}< 0,$$
majorer le nombre de côtés $k$ permettra de minorer l'aire $\area(D_0(g_p))$.\par

  L'image de $\partial D_0(\tilde g)$ dans $S$ forme le $1$-squelette d'une décomposition cellulaire de $S$. Cette décomposition se compose d'une face, de $e$ arêtes, et de $v$ sommets. Chaque sommet étant adjacent à au moins trois arêtes, nous avons $3v\leq2e$. Par conséquent $\chi(S)=e-v+1\leq -e/3+1$, ou de manière équivalente $e\leq 3(1-\chi(S))$.
Ceci permet de conclure car le nombre de côtés de $D_0(\tilde g)$ est égal à $2e$, et $D_0(g_p)$ a moins de côtés que $D_0(\tilde g)$.\end{proof}

\section{Deux applications immédiates du lemme de Schwarz}\label{sec:applications}

 Rappelons un théorème de Yamada (\cite{yamada,schwarz})~:
 
\begin{theoremnonumber}[Yamada] 
Si $\g$ est une géodésique fermée primitive non simple d'une surface hyperbolique, alors $\ell(\g)\geq 2\mathrm{arccosh}(3)$.
 Cette constante est optimale quel que soit le type topologique de la surface. De plus, elle n'est atteinte que pour le pantalon hyperbolique à trois pointes.
\end{theoremnonumber}

  Ce théorème a été étendu aux surfaces complètes à courbure $0>K\geq -1$ par P.~Buser (\cite{buser} théorème~4.3.1). Dans le paragraphe suivant la démonstration du théorème, Buser se demande si l'on peut supprimer l'hypothèse de courbure négative. Nous répondons par l'affirmative, en combinant le théorème de Yamada avec le lemme de Schwarz (\textsection~\ref{sec:schwarz}) nous obtenons immédiatement~: 

\begin{theorem4}
Soit $S$ une surface, éventuellement à bord non vide, dont le groupe fondamental n'est pas virtuellement abélien. Soient $g_0$ une métrique hyperbolique sur $S$, et $\g$ un lacet librement homotope à une géodésique primitive non simple de $g_0$. Si $g$ est une métrique complète sur $S$ à courbure $K(g)\geq -1$, alors $$\ell_g(\g)\geq 2\mrm{arccosh}(3).$$
 Cette constante est optimale quel que soit le type topologique de la surface. Si elle est atteinte, alors la surface est homéomorphe à la sphère privée de trois points.
\end{theorem4}

Dans le même esprit, nous avons le lemme suivant~:

\begin{lemma}
Soit $S$ une surface orientable dont les composantes de bord sont compactes, et dont le groupe fondamental n'est pas abélien. Soit $g$ une métrique complète sur $S$, à bord géodésique et à courbure $K(g)\geq -1$.
Si $\g$ et $\delta$ sont deux géodésiques fermées simples de $(S,g)$ qui s'intersectent en exactement un point, alors 
$$\sinh\left(\frac{\ell_g(\g)}{2}\right)\sinh\left(\frac{\ell_g(\d)}{2}\right)\geq 1.$$
\end{lemma}

\begin{proof}
Par le lemme~\ref{lem:technique}, nous nous ramenons au cas où $S$ est sans bord.
Comme $\g$ et $\d$ s'intersectent en exactement un point, elles sont non contractiles. Soit $g_0$ une métrique hyperbolique sur $S$ conforme à $g$. Il existe des $g_0$-géodésiques fermées simples $\g_0$ et $\delta_0$ isotopes à $\g$ et $\d$. Elles s'intersectent en un point, et par le lemme du collier (voir \cite{buser} \textsection~4), nous avons 
$\sinh\left(\ell_{g_0}(\g_0)/2\right)\sinh\left(\ell_{g_0}(\d_0)/2\right)\geq 1.$
Nous concluons grâce au lemme de Schwarz (voir \textsection~\ref{sec:schwarz}).
\end{proof}

 Le lemme ci-dessus peut être vu comme une conséquence du \emph{lemme du collier en courbure variable} de Buser (théorème~4.3.2 de \cite{buser}). Nous pouvons aussi prouver le théorème de Buser \emph{via} le lemme de Schwarz et le lemme~\ref{lem:technique}, sans recourir à des généralisations des identités trigonométriques (\cite{buser} \textsection~2.5).


\bibliographystyle{alpha}
\bibliography{biblio}


\end{document}